\documentclass[11pt]{article}

\usepackage{amsmath,amsthm,amssymb}
\usepackage{caption,graphicx,subfig,xcolor}
\usepackage[text={6in,8.5in},centering]{geometry}
\usepackage[sort,numbers]{natbib}
\usepackage{hyperref}
\hypersetup{colorlinks=true,linkcolor=black,urlcolor=blue,citecolor=blue}
\usepackage[capitalize]{cleveref}
\usepackage{multirow,floatrow}

\usepackage{float}
\floatstyle{plaintop}
\restylefloat{table}

\crefname{section}{section}{sections}
\crefname{subsection}{subsection}{subsections}
\crefname{figure}{figure}{figures}
\Crefname{section}{Section}{Sections}
\Crefname{subsection}{Subsection}{Subsections}
\Crefname{figure}{Figure}{Figures}

\usepackage[sort,numbers]{natbib}
\makeatletter\@addtoreset{equation}{section}\makeatother

\newtheorem{thm}{Theorem}[section] 
\newtheorem{lem}[thm]{Lemma}

\newtheoremstyle{named}{}{}{\itshape}{}{\bfseries}{.}{.5em}{\thmnote{#3's }#1}
\theoremstyle{named}

\theoremstyle{definition}

\newcommand{\R}{\mathbb{R}}
 
\newcommand{\drm}{\mathrm{d}}
\newcommand{\uv}{{\bf u}}
\newcommand{\zv}{{\bf z}}
\newcommand{\bF}{{\bf F}}

\newcommand{\ddxi}{\frac{\rm d}{{\rm d}\xi}}
\newcommand{\cS}{U_1}
\newcommand{\cP}{U_2}
\newcommand{\zDomain}{Z}
\newcommand{\eps}{\varepsilon}

\crefformat{equation}{\textup{#2(#1)#3}}
\crefrangeformat{equation}{\textup{#3(#1)#4--#5(#2)#6}}
\crefmultiformat{equation}{\textup{#2(#1)#3}}{ and \textup{#2(#1)#3}}
{, \textup{#2(#1)#3}}{, and \textup{#2(#1)#3}}
\crefrangemultiformat{equation}{\textup{#3(#1)#4--#5(#2)#6}}%
{ and \textup{#3(#1)#4--#5(#2)#6}}{, \textup{#3(#1)#4--#5(#2)#6}}{, and \textup{#3(#1)#4--#5(#2)#6}}
\Crefformat{equation}{#2Equation~\textup{(#1)}#3}
\Crefrangeformat{equation}{Equations~\textup{#3(#1)#4--#5(#2)#6}}
\Crefmultiformat{equation}{Equations~\textup{#2(#1)#3}}{ and \textup{#2(#1)#3}}
{, \textup{#2(#1)#3}}{, and \textup{#2(#1)#3}}
\Crefrangemultiformat{equation}{Equations~\textup{#3(#1)#4--#5(#2)#6}}%
{ and \textup{#3(#1)#4--#5(#2)#6}}{, \textup{#3(#1)#4--#5(#2)#6}}{, and \textup{#3(#1)#4--#5(#2)#6}}

\newcommand\solidrule[1][15pt]{\rule[0.5ex]{#1}{1pt}}
\newcommand\dashedrule{\mbox{%
	\solidrule[3pt]\hspace{3pt}\solidrule[3pt]\hspace{3pt}\solidrule[3pt]}}

\newcommand\dottedrule{\mbox{%
	\solidrule[1.5pt]\hspace{2pt}\solidrule[1.5pt]\hspace{2pt}\solidrule[1.5pt]\hspace{2pt}\solidrule[1.5pt]\hspace{2pt}\solidrule[1.5pt]}}
\newcommand\dashdottedrule{\mbox{%
	\solidrule[4.5pt]\hspace{2pt}\solidrule[2pt]\hspace{2pt}\solidrule[4.5pt]}}

\definecolor{mygreen}{rgb}{0,0.4,0}
\definecolor{mypink}{rgb}{.9,.4,.7}

\begin{document}

\title{Minimum wave speeds in monostable reaction--diffusion equations: sharp bounds by polynomial optimization}

\author{
Jason J. Bramburger\thanks{Current e-mail: jbrambur@uw.edu}\ \ and David Goluskin\\[6pt]
Department of Mathematics and Statistics\\
University of Victoria \\
Victoria, BC, V8P 5C2, Canada
}

\date{}
\maketitle

\begin{abstract}
Many monostable reaction--diffusion equations admit one-dimensional travelling waves if and only if the wave speed is sufficiently high. The values of these minimum wave speeds are not known exactly, except in a few simple cases. We present methods for finding upper and lower bounds on minimum wave speed. They rely on constructing trapping boundaries for dynamical systems whose heteroclinic connections correspond to the travelling waves. Simple versions of this approach can be carried out analytically but often give overly conservative bounds on minimum wave speed. When the reaction--diffusion equations being studied have polynomial nonlinearities, our approach can be implemented computationally using polynomial optimization. For scalar reaction--diffusion equations, we present a general method and then apply it to examples from the literature where minimum wave speeds were unknown. The extension of our approach to multi-component reaction--diffusion systems is then illustrated using a cubic autocatalysis model from the literature. In all three examples and with many different parameter values, polynomial optimization computations give upper and lower bounds that are within 0.1\% of each other and thus nearly sharp. Upper bounds are derived analytically as well for the scalar RD equations.
\end{abstract}


\section{Introduction}

Reaction--diffusion (RD) equations are a large family of partial differential equations (PDEs) modeling diverse phenomena that include chemical reactions, spatial effects in ecology, morphogenesis, and interacting particle systems~\cite{Cantrell, Cremer, Masi, Murray, Turing, Volpert}. Among the simplest solutions to RD equations is the travelling wave---a state that is time-independent in a reference frame that moves at the wave speed. Such states arise in models of applications that include biological invasion~\cite{Cantrell,Ding}, urban crime~\cite{Berestycki}, and bacterial pattern formation~\cite{Mansour,Satnoianu}. In many systems whose reaction dynamics are monostable, there is a finite minimum wave speed such that travelling waves exist at all faster speeds and no slower ones \cite{Saarloos}. In all but the simplest cases, values of the minimum speeds are not known precisely~\cite{Billingham, Hosono2, Kollar, Li, Mansour, Mingxin, Saarloos, Takase}. These values are of particular importance for monostable RD equations where localized initial conditions lead to fronts travelling at the minimum wave speeds. This property has been proved to hold for various RD equations~\cite{Aaronson, KPP, Saarloos, Stokes}, including with stochastic forcing~\cite{Mueller}, but there are many other RD equations where this property is suspected but not proved. The latter cases call for a way to determine minimum wave speeds that does not rely on numerical integration of PDEs. Furthermore, many RD equations are computationally expensive to integrate due to stuffness, so it would be valuable to have a method for finding minimum wave speeds that avoids integrating PDEs.

Here we present methods for precisely estimating minimum speeds of one-dimensional travelling waves in monostable RD equations. The methods apply to scalar equations broadly and to certain classes of multicomponent RD systems. In the scalar case we consider RD equations of the form
\begin{equation}
\label{eq: general pde}
u_t = (D(u)u_x)_x + a(u)u_x + f(u),
\end{equation}
where the reaction term $f$ is monostable, and $x$ is the single coordinate in which the wave varies. (Multicomponent systems are discussed in below.) Our approach relies on the fact that a one-dimensional travelling wave in the RD equation~\cref{eq: general pde}, or its multicomponent analogue, is equivalent to a heteroclinic connection in a corresponding system of ordinary differential equations (ODEs). A travelling wave solution of~\cref{eq: general pde} takes the form $u(x,t) = u(\xi)$, where $\xi:=x-ct$ for some wave speed $c$, in which case $u(\xi)$ must solve the ODE
\begin{equation}
\label{eq: general ode}
(D(u)u_\xi)_\xi + (a(u) + c)u_\xi + f(u) = 0.
\end{equation}
Introducing a vector variable such as $\uv=(u,u_\xi)$ lets the above ODE be written as an autonomous dynamical system,
\begin{equation}
\label{eq: dyn syst}
\dot\uv=\bF(\uv;c),
\end{equation}
where $\dot\uv$ denotes $\tfrac{\rm d}{{\rm d}\xi}\uv$, and $c$ is a parameter. A dynamical system~\cref{eq: dyn syst} governing travelling waves will be two-dimensional in the case of scalar RD equations~\cref{eq: general pde} and higher-dimensional in the case of multicomponent RD systems. In either case a heteroclinic trajectory $\uv(\xi)$ of~\cref{eq: dyn syst}, going from a \emph{source} equilibrium to a \emph{target} equilibrium, corresponds to a travelling wave in the original RD equation. To seek travelling waves that satisfy additional constraints, such as nonnegativity of $u(x,t)$, corresponding geometric constraints can be placed on the heteroclinic connection $\uv(\xi)$. For an RD equation with a unique minimum wave speed $c_*$, if the corresponding ODE system can be shown to have a heteroclinic connection for some value of $c$, then $c$ is an upper bound on $c_*$. If the ODE system can be shown to lack such a connection, then $c$ is a lower bound on $c_*$. What is therefore needed are methods to determine whether or not two equilibria in an ODE system have a trajectory connecting them, possibly subject to further constraints on that trajectory. The present work provides such methods, which in turn give precise upper and lower bounds on $c_*$.

The existence or nonexistence of a heteroclinic connection between specified source and target equilibria in $n$-dimensional phase space can be determined by finding certain $(n-1)$-dimensional surfaces that are trapping boundaries, meaning trajectories cross these surfaces in only one orientation. To show that a connection \emph{does not} exist, it suffices to find trapping surfaces that form a barrier between the source and the target. To show that a heteroclinic connection \emph{does} exist, it often is useful to show that some trajectories leaving the source equilibrium must enter a trapping region containing the target, but this alone is not enough. Verifying the connection requires a second argument, often relying on some type of monotonicity. If the target equilibrium attracts all trajectories within the trapping region, for instance, this can be shown by constructing a Lyapunov function which decreases monotonically along trajectories. (If the target is a saddle, one might employ Wazewski's theorem instead of monotonicity \cite[Proposition~2]{Dunbar2}.) Suitable monotonicity can be difficult to show for general ODE systems, but the task is relatively easy for ODEs that correspond to travelling waves in monostable RD equations. For each RD equation studied here, the corresponding ODE system has a locally attracting target and a simple monotonicity property. In such cases the main difficulty in showing existence of a heteroclinic connection, as in showing nonexistence, is to find suitable trapping boundaries.

When seeking a trapping surface whose geometry implies the existence or nonexistence of a heteroclinic connection, we impose constraints on the surface to ensure it has the desired implication. Examples throughout the present work illustrate such constraints, but choosing them is not the main difficulty. The difficulty, at least historically, is to find a surface that satisfies given constraints and is indeed a trapping boundary. Trapping surfaces typically have been found in an \emph{ad hoc} way. A common analytical approach is to consider a simple linear or quadratic surface, possibly with free parameters that can be tuned, and try to show that trajectories cross this surface in only one orientation. At the core of our present contribution is a powerful computational method for finding trapping surfaces defined by much more general polynomial expressions.

We construct trapping boundaries in the phase space of dynamical systems using two different approaches, which we refer to as the volume method and the surface method. Both require finding \emph{auxiliary functions} that satisfy certain inequalities pointwise on a suitable part of phase space. In principle such functions can be sought analytically or computationally, but in many cases the auxiliary functions must be complicated, especially when $c$ is close to the minimum wave speed $c_*$, so simple analytical constructions are unable to give precise upper or lower bounds on $c_*$. An important case where computational methods can be used to construct more complicated auxiliary functions is that of spatially homogenous RD equations whose nonlinearities are polynomial. All examples studied here are of this type.

When all nonlinearities of an RD equation are polynomial, the state vector $\uv$ can be defined such that the dynamical system~\cref{eq: dyn syst} which governs travelling waves has a right-hand side that is polynomial in the components of $\uv$. We then restrict our search for auxiliary functions to some finite-dimensional set of polynomials. In such cases, the inequalities that auxiliary functions must satisfy in order to give trapping boundaries are equivalent to various polynomial expressions being nonnegative on specified parts of phase space. Although deciding nonnegativity of a multivariable polynomial is NP-hard in general~\cite{Murty1987}, nonnegativity can be replaced by the stronger but more tractable condition that the polynomial admits a representation as a sum of squares (SOS) of other polynomials. The resulting SOS-constrained problems can be solved computationally by an approach that has become widely used since its introduction two decades ago~\cite{Nesterov2000, Parrilo2000, Lasserre2001}: reformulate the SOS problem as a semidefinite program~\cite{Boyd2004} (a standard type of convex optimization problem) and solve the latter numerically. In the present context these numerics give an auxiliary function satisfying the imposed SOS constraints, provided any such function exists in the set of polynomials where it is sought. Depending on how the SOS constraints are designed, success in finding an auxiliary function at some value of $c$ implies either existence or nonexistence of a heteroclinic connection in the dynamical system~\cref{eq: dyn syst} and, in turn, existence or nonexistence of travelling waves in an RD equation.

The methods presented here join a number of recent works in which SOS-constrained computations are used to study dynamical systems, although none have studied travelling waves specifically. Various information about dynamical systems can be inferred by constructing auxiliary functions that satisfy various constraints. The best-known type of auxiliary function is the Lyapunov function, whose constraints imply nonlinear stability of a particular state. These can be constructed by SOS computations~\cite{Parrilo2000, Papachristodoulou2002}, as can other types of auxiliary functions that provide bounds on deterministic or stochastic time averages~\cite{Chernyshenko2014, Fantuzzi2016siads, Kuntz2016, Goluskin2016lorenz, Tobasco2018, Korda2018a, Goluskin2019}, on extrema over global attractors~\cite{Goluskin2020a}, and on transient extrema~\cite{Fantuzzi2020}. Especially related to our approach are SOS methods providing guarantees that trajectories do~\cite{Magron2019, Jones2019} or do not~\cite{Prajna2007, Ahmadi2017} enter specified sets. The present study adapts these ideas specifically to ODEs that govern travelling waves in RD equations, where the particular structure of such ODEs allows for novel methods that exploit~it.

We have applied the general approach presented here to compute upper and lower bounds on the minimum wave speed $c_*$ for several examples in which $c_*$ is not known exactly, including two scalar RD equations and a two-component RD system. The first scalar example is a generalized Fisher--KPP equation~\cite{Fisher, KPP, Billingham}, $u_t = u_{xx} + u^m(1-u)$, for integers $m$ larger than the well understood value $m=1$, and the second is a chemotaxis model introduced in~\cite{Mansour}. Our two-component example is a model of cubic autocatalysis that has been studied for decades~\cite{Billingham1991a, Billingham1991b, Chen, Chen2, Focant, Hosono, Qi}. In all three examples, our best bounds at many different parameter values are within 0.1\% of being sharp, as indicated by the near-equality of upper and lower bounds. 

This manuscript is organized as follows. Section~\ref{sec:Methods} presents two general approaches for finding trapping boundaries in dynamical systems using auxiliary functions, as well as how to construct these functions computationally using SOS constraints when the differential equations are polynomial. In \Cref{sec:Single} our framework is tailored to scalar RD equations in general, and then it is applied both computationally and analytically to produce bounds on $c_*$ for our two examples of this type. \Cref{sec:Multi} extends these methods to our chosen example of a two-component RD system, for which we compute bounds and infer asymptotic scalings of $c_*$. \Cref{sec:Con} offers conclusions.


\section{\label{sec:Methods}Constructing trapping boundaries using auxiliary functions}

In this section we describe methods for finding trapping boundaries in the phase spaces of dynamical systems. Variants of these methods appear widely in the literature, although usually without assistance from SOS computations. Only the applications of these methods in subsequent sections are specific to dynamical systems that govern travelling waves in RD equations. For such applications, the examples of Sections~\ref{sec:Single} and~\ref{sec:Multi} illustrate how to constrain the trapping boundaries being constructed so that they imply the existence or nonexistence of heteroclinic connections. First we consider how to find trapping boundaries in general, subject to constraints on their locations in phase space. The approach of Section~\ref{subsec:Volume}, which we call the \emph{volume method}, seeks an auxiliary function that satisfies an inequality pointwise on a positive-volume subset of the phase space $\R^n$. The approach of Section~\ref{subsec:Surface}, which we call the \emph{surface method}, instead seeks an auxiliary function that satisfies an inequality on an $(n-1)$-dimensional surface, but it can be applied only to dynamical systems with certain structure. The volume and surface inequalities each define convex sets of auxiliary functions. This convexity is exploited in Section~\ref{subsec:SOS}, which describes how suitable functions can be constructed computationally using SOS constraints, provided the right-hand side of the dynamical system~\cref{eq: dyn syst} is polynomial in $\uv$.


\subsection{\label{subsec:Volume}Volume inequality conditions}

Consider a well-posed autonomous dynamical system~\cref{eq: dyn syst} with trajectories $\uv(\xi)$ in $\R^n$ and all parameters fixed. We aim to find a continuously differentiable auxiliary function $V:\R^n\to\R$ whose zero level set forms a trapping boundary, at least within some specified region $U\subseteq\R^n$. The $V(\uv)=0$ level set divides $U$ according to the sign of $V$. Trajectories remaining in $U$ are trapped in, say, the negative-$V$ region if and only if $\ddxi V(\uv(\xi))\le 0$ for all $\uv(\xi)$ on the zero-$V$ set. Because $\ddxi V(\uv(\xi))=\bF(\uv(\xi))\cdot\nabla V(\uv(\xi))$ along every trajectory, where $\bF$ is the right-hand side of the dynamical system~\cref{eq: dyn syst}, the trapping condition can be stated independently of $\xi$ as
\begin{equation}
\label{eq: V surface}
\bF(\uv)\cdot\nabla V(\uv) \leq 0 \quad \forall~\uv\in U~~s.t.~~V(\uv)=0.
\end{equation}
We ultimately want to search over a large class of functions for a $V$ that satisfies~\cref{eq: V surface}. It is prohibitively hard to do this directly because the set of $V$ satisfying~\cref{eq: V surface} would form a non-convex set, due to the appearance of $V$ both in the inequality and in the $V(\uv)=0$ condition. Instead we impose an inequality over all of $U$ that is sufficient for~\cref{eq: V surface}: we require that there exists a constant $\lambda>0$ such that 
\begin{equation}
\label{eq: Trapping}
\lambda\, \bF(\uv)\cdot\nabla V(\uv) \leq - V(\uv) \quad \forall~\uv\in U.
\end{equation}
Although this condition does not imply $\bF(\uv)\cdot\nabla V(\uv) \leq 0$ on all of $U$, it does so on the subset of $U$ where $V(\uv)$ vanishes, thereby implying~\cref{eq: V surface}. Crucially, for any convex class of functions, the subset of functions satisfying~\cref{eq: Trapping} is convex. This leads to tractable computational methods for constructing $V$, as described below in Section~\ref{subsec:SOS}. One might imagine that imposing~\cref{eq: Trapping} over all of $U$ is too restrictive, but the successful computations reported in Sections~\ref{sec:Single} and~\ref{sec:Multi} show otherwise. We refer to the use of condition~\cref{eq: Trapping} as the volume method since the inequality is imposed on a set $U$ of nonzero volume, in contrast to the approach of the next subsection.

To use the volume method to verify existence or nonexistence of heteroclinic connections, the volume condition~\cref{eq: Trapping} must be accompanied by other constraints that are specific to the dynamical system under study. Suppose all trajectories of interest lie in a fixed set $U$, the choice of which can encode desired restrictions such as nonnegativity of components of $\uv(\xi)$. To show nonexistence of a heteroclinic connection within $U$, from a specified source $\uv^-$ and target $\uv^+$, it suffices to find a $V$ that satisfies~\cref{eq: Trapping} as well as the scalar constraints $V(\uv^-)\le 0$ and $V(\uv^+)=1$. Auxiliary functions used in this way, to create a trapping barrier between specified regions, are sometimes called barrier functions in the literature. A few recent works have sought barrier functions computationally using SOS methods, as we do here; see \cite{Prajna2007, Ahmadi2017} and references therein. To show existence of a heteroclinic connection within $U$, various constraints must be imposed alongside~\cref{eq: Trapping}, depending on the geometry of phase space. The choice of such constraints is illustrated in Sections~\ref{sec:Single} and~\ref{sec:Multi}.


\subsection{\label{subsec:Surface}Surface inequality conditions}

We now describe a way to construct trapping boundaries using auxiliary functions that are constrained only on the boundary set, in contrast to the volume method of the previous subsection. For general dynamical systems, a volume condition such as~\cref{eq: Trapping} is needed in order for the auxiliary functions satisfying the constraint to form a convex set, as explained above. When the dynamical system has a particular structure, however, there is a convex way to constrain auxiliary functions only on an $(n-1)$-dimensional trapping surface. We refer to this approach as the surface method. It is useful to us because travelling waves in RD equations often correspond to dynamical systems with the requisite structure. 

We seek a trapping surface that can be described as a graph. Suppose for concreteness that the last coordinate of $\uv$ on this surface can be given as a function of the others, meaning $u_n=N(\zv)$, at least for all $\zv=(u_1,\ldots,u_{n-1})$ in some $\zDomain\subseteq\R^{n-1}$. This surface divides $\zDomain\times\R$ in two based on the sign of $u_n-N$. To show that it traps trajectories in, say, the region where $u_n\ge N$, one must show that $\ddxi[u_n(\xi)-N(\zv(\xi))]\ge0$ at all points where a trajectory $\uv(\xi)=(\zv(\xi),u_n(\xi))$ intersects the surface. Applying the chain rule and the fact that $u_n=N(\zv)$ gives the equivalent condition
\begin{equation}
\label{eq: SurfaceCond}
\begin{bmatrix}-\nabla N(\zv)\\1\end{bmatrix}\cdot \bF(\zv,N(\zv)) \geq 0
\quad \forall~\zv\in \zDomain,
\end{equation}
where $\bF$ is the right-hand side of the dynamical system~\cref{eq: dyn syst}. For general $\bF$, the functions $N$ satisfying~\cref{eq: SurfaceCond} will not form a convex set, in which case we cannot construct $N$ by the computational methods described in the next subsection. An important exception occurs when $\bF$ depends only linearly on $u_n$ in the manner
\begin{equation}
\label{eq: SurfaceODE}
\bF(\uv) = \begin{bmatrix} \bF_1(\zv) \\ F_2(\zv)+F_3(\zv)u_n \end{bmatrix} + u_n\hat e_{n-1},
\end{equation}
where $\hat e_{n-1}$ is the unit vector in the $u_{n-1}$ direction, $F_2,F_3$ are scalar-valued functions, and \mbox{$\bF_1:\R^{n-1}\to\R^{n-1}$}. Such ODEs arise, for example, from multicomponent RD systems of the form ${\bf w}_t = {\bf w}_{xx} + G({\bf w})$, where ${\bf w}$ is a vector. Travelling waves in such systems must solve ${\bf w}_{\xi\xi} + c{\bf w}_\xi + G({\bf w}) = 0$, and letting $\uv = ({\bf w},{\bf w}_\xi)$ gives a first order system of ODEs in the form~\cref{eq: SurfaceODE}, where any component of ${\bf w}_\xi$ can take the place of $u_n$ in the above formulation. 

When $\bF$ has the form~\cref{eq: SurfaceODE}, the condition~\cref{eq: SurfaceCond} for the $u_n=N(\zv)$ surface to be trapping becomes
\begin{equation}
\label{eq: surf non-convex}
F_2(\zv)+F_3(\zv)N(\zv)-\bF_1(\zv)\cdot\nabla N(\zv) - \frac{\partial N}{\partial u_{n-1}}(\zv)N(\zv) \ge 0
\quad \forall~\zv\in \zDomain.
\end{equation}
Anticipating the computational formulation in the next subsection, we transform~\cref{eq: surf non-convex} into a constraint that is linear in $N$. First we take an antiderivative with respect to the $u_{n-1}$ coordinate. We assume for simplicity that the domain $\zDomain$ has a boundary on the $u_{n-1}=0$ hyperplane, and that $N$ vanishes there. Under these assumptions, which are not needed but do hold in our examples of \cref{sec:Single}, expression~\cref{eq: surf non-convex} gives
\begin{equation}
\label{eq: Surf1}
\int_0^{u_{n-1}} [F_2(\zv)+F_3(\zv)N(\zv)- \bF_1(\zv)\cdot\nabla N(\zv)]\drm s - \frac{1}{2}N(\zv)^2 \geq 0 \quad \forall~\zv\in \zDomain,
\end{equation}
where the integration variable $s$ denotes the last component of $\zv$. The left-hand expression in~\cref{eq: Surf1} is not linear in $N$ but is the Schur complement of a $2\times2$ matrix, so~\cref{eq: Surf1} is equivalent to semidefiniteness of that matrix:
\begin{equation}
\begin{bmatrix}
\label{eq: N matrix}
\int_0^{u_{n-1}} [F_2(\zv)+F_3(\zv)N(\zv)- \bF_1(\zv)\cdot\nabla N(\zv)]\ \drm s & N(\zv) \\
N(\zv) & 2	
\end{bmatrix} \succeq 0 \quad \forall~\zv\in \zDomain.
\end{equation}
This constraint is linear in $N$, and it can be rewritten as a scalar inequality using the definition of semidefiniteness. Multiplying on the left and right by an arbitrary vector $(y_1,y_2)$ gives an equivalent scalar inequality that also is linear in $N$:
\begin{multline}
\label{eq: Nineq}
y_1^2 \int_0^{u_{n-1}} [F_2(\zv)+F_3(\zv)N(\zv)- \bF_1(\zv)\cdot\nabla N(\zv)]\ \drm s + 2y_1y_2N(\zv) + 2y_2^2 \geq 0 \\ \forall~(\zv,y_1,y_2)\in \zDomain\times \R^2.
\end{multline}
Finding any auxiliary function $N$ that satisfies~\cref{eq: Nineq} shows that the surface $u_n=N(\zv)$ traps trajectories in the region where $u_n\ge N(\zv)$. Computational methods for constructing $N$ subject to~\cref{eq: Nineq} are described in the next subsection. In \Cref{sec:Single} we also apply condition~\cref{eq: Nineq} analytically in the simple case of linear $N$.

Using the surface method to infer the existence or nonexistence of heteroclinic connections requires combining~\cref{eq: Nineq} with additional constraints on $N$. \Cref{sec:Single} gives such constraints for general ODEs corresponding to travelling waves in scalar RD equations. For these equations we use the surface method to show existence but are unable to use it to show nonexistence. Therefore the surface method here provides \emph{only upper bounds} on the minimum wave speed $c_*$, whereas the volume method gives both upper and lower bounds. \Cref{sec:Single} reports upper bounds found computationally using polynomial $N$ of various degrees, as well as upper bounds derived analytically using linear $N$.


\subsection{\label{subsec:SOS}Computational formulation with sum-of-squares polynomial constraints}

The sufficient conditions for finding trapping boundaries described in sections \ref{sec:Methods}\ref{subsec:Volume} and \ref{sec:Methods}\ref{subsec:Surface} require finding auxiliary functions that satisfy inequalities pointwise on specified sets. In particular, $V$ must satisfy~\cref{eq: Trapping} on $U\subset \R^n$, and $N$ must satisfy~\cref{eq: Nineq} on $\zDomain\times\R^2\subset \R^{n+1}$. There is no fully general method for constructing functions subject to such inequalities, but a computational approach is possible when the ODE system's right-hand side $\bF(\uv)$ is polynomial in the components of $\uv$. All examples studied below are of this type since they correspond to RD equations that are polynomial in the PDE variables and their spatial derivatives. We henceforth assume that $\bF(\uv)$ is polynomial, and we seek auxiliary functions, $V(\uv)$ or $N(\zv)$, that are polynomial also. To make this a finite-dimensional search, one must choose a finite polynomial vector space in which to seek $V$ or $N$, for instance by specifying a maximum polynomial degree of the auxiliary function.

Since we restricted attention to polynomial ODEs and polynomial auxiliary functions, the volume condition~\cref{eq: Trapping} and surface condition~\cref{eq: Nineq} each amount to nonnegativity of a polynomial expression on a specified set. Deciding nonnegativity of a polynomial over $\R^n$ has prohibitive computational complexity unless the polynomial degree is very small. A stronger condition that is easier to check is that the polynomial belongs to $\Sigma_n$, the set of polynomials in $n$ variables that admit SOS representations. For concreteness consider the volume condition~\cref{eq: Trapping}, where the polynomial $-(\lambda\, \bF\cdot\nabla V + V)$ must be nonnegative on $U$. Requiring this expression to belong to $\Sigma_n$ would imply nonnegativity on $U$ but is overly strong since it also implies nonnegativity globally on $\R^n$. Instead, there is a standard way to formulate SOS conditions that imply nonnegativity on $U$ but not on $\R^n$~\cite{Parrilo2013a, Lasserre2015}. Assume that the set $U$ is semialgebraic, meaning it can be specified by a finite number of polynomial inequalities and equalities:
\begin{equation}
\label{eq: U}
U = \{\uv\in\R^n :\ s_i(\uv)\ge0\text{ for }i=1,\ldots,I,~r_j(\uv)=0\text{ for }j=1,\ldots,J\}.
\end{equation}
A sufficient condition for the volume constraint~\cref{eq: Trapping} is the existence of polynomials $\{\sigma_i(\uv)\}_{i\le I}$ and $\{\rho_j(\uv)\}_{j\le J}$ such that the following SOS conditions are satisfied:
\begin{subequations}
\label{eq: p sos}
\begin{align}
-(\lambda\, \bF\cdot\nabla V + V) - \sum\limits_{i=1}^Is_i\sigma_i - \sum\limits_{j=1}^Jr_j\rho_j \in\Sigma_n, & \label{eq: p sos 1} \\
\sigma_i\in\Sigma_n & \quad \text{for }i=1,\ldots,I. \label{eq: p sos 2}
\end{align}
\end{subequations}
To see that these conditions imply~\cref{eq: Trapping}, note that on $U$ (but not necessarily outside of $U$) the first sum in~\cref{eq: p sos} is nonnegative and the second sum vanishes. Such a use of SOS conditions to show nonnegativity on a semialgebraic set is variously called a weighted SOS condition~\cite{Lasserre2015} or an $\mathcal S$-procedure \cite{Tan2006, Fantuzzi2016siads}. In our examples below, regions of phase space are specified using only polynomial inequalities, as opposed to equalities, so our formulations use SOS polynomials like $\sigma_i$ but no arbitrary polynomilas like $\rho_j$.

To search computationally for $V$, $\sigma_i$, and $\rho_i$ satisfying SOS constraints such as~\cref{eq: p sos}, one represents each unknown polynomial in a chosen finite basis with tunable coefficients. These tunable variables appear linearly in the SOS constraints, which is essential for translating the SOS-constrained problem into a semidefinite program for computational solution. This is why we have insisted that the auxiliary functions appear only linearly in the inequality constraints derived in Sections~\ref{subsec:Volume} and~\ref{subsec:Surface}. Any further constraints on $V$ that are added to infer existence or nonexistence of a heteroclinic connection, such as those described in Section~\ref{subsec:Volume}, must also be linear in the tunable coefficients. Likewise, because $\lambda$ multiplies the tunable coefficients of $V$ in~\cref{eq: p sos}, its value must be fixed in each SOS computation. If desired, one can optimize over $\lambda$ by solving an SOS problem repeatedly while sweeping through $\lambda$. For further details about SOS conditions and their conversion to semidefinite programs, we refer the reader to the review~\cite{Parrilo2013a}.

After choosing vector spaces in which to seek $V$, $\sigma_i$, and $\rho_i$, the semidefinite program corresponding to~\cref{eq: p sos} is solved as a feasibility problem, meaning there is no optimization objective. The computation may either find admissible values of the tunable coefficients or verify that no such values exist. In the computational examples reported below we used the MATAB software YALMIP (version R20190425) \cite{Lofberg2004, Lofberg2009} to translate SOS constraints into semidefinite programs, and we solved the latter using Mosek (version 9.0) \cite{Mosek}. All code is publicly available at the repository {\bf GitHub/jbramburger/Minimum-Wave-Speeds}.

The surface condition~\cref{eq: Nineq} can be enforced by SOS constraints very similarly to the above-described way in which the volume condition~\cref{eq: Trapping} is enforced by~\cref{eq: p sos}, at least when the domain $\zDomain$ is a semialgebraic set. One additional step is needed because of the integral in~\cref{eq: Nineq}: once $N$ is expanded in a chosen finite basis with tunable coefficients, the integration in~\cref{eq: Nineq} must be carried out symbolically. The resulting expression is a polynomial in $(\zv,y_1,y_2)$ in which tunable coefficients still appear linearly, and its nonnegativity on $\zDomain\times\R^2$ can be enforced by SOS constraints analogous to~\cref{eq: p sos}. Additional constraints can be added to help infer the existence of heteroclinic connections, provided they too are linear in $N$.

Our focus in subsequent sections is on dynamical systems whose heteroclinic connections correspond to travelling waves in monostable RD equations. When showing the existence or nonexistence of connections, we want to do so for the smallest and sufficiently large wave speeds, respectively. Unfortunately we cannot specify $c$ as the optimization objective in an SOS computation subject to constraints such as~\cref{eq: p sos}. Since $c$ is multiplied by tunable coefficients in the SOS constraints, $c$ cannot be tunable itself, so we must perform SOS computations at multiple fixed values of $c$. For instance, to find the smallest $c$ value at which the existence of travelling waves can be verified using~\cref{eq: p sos}, we begin with a large value of $c$ at which the SOS computation is feasible and a small value of $c$ at which it is infeasible. We then bisect repeatedly in $c$, solving an SOS program at each new value, to find the smallest $c$ at which the computation is feasible---the best upper bound on $c_*$. (The numerical result is often slightly conservative because, when $c$ is very close to the value that separates feasibility and infeasibility, the semidefinite program solver often can verify neither feasibility nor infeasibility.) To find the best lower bound on $c_*$, we carry out an analogous bisection search in $c$ to find the largest value at which an SOS problem verifying nonexistence is feasible. To try to improve the resulting upper and lower bounds on $c_*$, one can enlarge the vector spaces over which the tunable polynomials are sought and repeat the entire procedure. Although we do not have a proof that existence and non-existence of heteroclinic orbits can be verified for arbitrarily sharp $c$, we note that it appears to be true in all examples presented here.

In the next two sections we report, for various RD equations, upper and lower bounds on $c_*$ computed using the SOS relaxation~\cref{eq: p sos} of the volume method~\cref{eq: Trapping}. For the scalar RD examples in \Cref{sec:Single} we report upper bounds on $c_*$ computed using the analogous SOS relaxation of the surface method~\cref{eq: Nineq}. Upper bounds on $c_*$ computed using both methods are compared for the example of Section~\ref{subsec: Fisher}, in which case the surface method gives sharper bounds with less effort.


\section{\label{sec:Single}Scalar RD equations}

In this section we specialize the methods of Section~\ref{sec:Methods} to estimate minimum speeds $c_*$ at which scalar RD equations admit travelling waves that are \emph{monotone}, meaning solutions $u(\xi)$ such that $u_\xi\le0$ for all $\xi:=x-ct\in\R$. We consider general scalar RD equations~\cref{eq: general pde} under the assumptions that $D(u) > 0$ for all $u \in (0,1]$, $f(u) > 0$ for all $u \in (0,1)$, and $f(0)=f(1) = 0$. Monotone travelling waves in such equations satisfy the second-order ODE~\cref{eq: general ode}, obey $0\le u(\xi)\le1$, and approach the limits $u(-\infty) = 1$ and $u(\infty) = 0$. Conditions giving upper and lower bounds on $c_*$ are formulated below in Subsections~\ref{subsec:1DUpper} and~\ref{subsec:1DLower}, respectively. These conditions can be implemented using SOS computations, provided that $D(u)$ and $f(u)$ are polynomials. Subsections~\ref{subsec: Fisher} and \ref{subsec: chemotaxis} report results of SOS computations for a generalized Fisher--KPP equation and a chemotaxis model, respectively, as well as upper bounds on $c_*$ derived analytically.

To apply the methods of Section~\ref{sec:Methods} we must rewrite the ODE~\cref{eq: general ode}, which governs travelling waves, as a first-order system. We want the right-hand side of this system to be polynomial whenever $D(u)$ and $f(u)$ are, which precludes using the phase space $(u, u_\xi)$. Instead we use $(u,v)$ with $v = D(u)u_\xi$, and following~\cite{Mansour} we define a new independent variable $\hat \xi$ by
\begin{equation}
	\xi = \int_0^{\hat{\xi}} \frac{\drm s}{D(u(s))}.
\end{equation}
In these variables, the ODE~\cref{eq: general ode} governing travelling waves becomes the dynamical system
\begin{equation}
\label{eq: scalar dyn syst}
	\begin{split}
		\dot{u} &= v, \\
		\dot{v} &= -[c + a(u)]v - D(u)f(u),
	\end{split}	
\end{equation}
where for the remainder of this section a dot denotes differentiation with respect to $\hat \xi$.

A travelling wave in the scalar RD equation corresponds to a heteroclinic connection in the dynamical system~\cref{eq: scalar dyn syst} from the source equilibrium $(1,0)$ to the target $(0,0)$. The constraint that the wave is monotonic corresponds to the requirement that the heteroclinic connection lies in the region
\begin{equation}
	\label{eq: S}
	\cS = \{(u,v)\in\R^2:\ 0\le u\le1,~v \leq 0\}.
\end{equation}
Note that $\cS$ can be defined in the form~\cref{eq: U} of a semialgebraic set by letting $s_1(u,v)=u(1-u)$ and $s_2(u,v)=-v$. The linearization of~\cref{eq: scalar dyn syst} at $(1,0)$ indicates that this equilibrium always has a one-dimensional unstable manifold entering $\cS$, which is the only trajectory that might connect to $(0,0)$. No trajectory can leave $\cS$ across the $v=0$ or $u=1$ boundaries, as follows from the form of~\cref{eq: scalar dyn syst} and the positivity of $D$ and $f$, so the unstable manifold of $(1,0)$ either connects to $(0,0)$ or leaves $\cS$ elsewhere on the $u=0$ boundary. A trapping boundary with suitable geometry can distinguish between these two possibilities; existence or nonexistence of a connection can be verified by finding boundaries like the ones sketched in \cref{fig:TrapRegion,fig:Barrier}, respectively. We seek such boundaries using the general methods of Section~\ref{sec:Methods}, whose formulations for the dynamical system~\cref{eq: scalar dyn syst} are made explicit in the next two subsections.

\floatsetup[figure]{style=plain,subcapbesideposition=top}	
\begin{figure}[tp] 
\centering
\sidesubfloat[]{\includegraphics[width=0.4\textwidth]{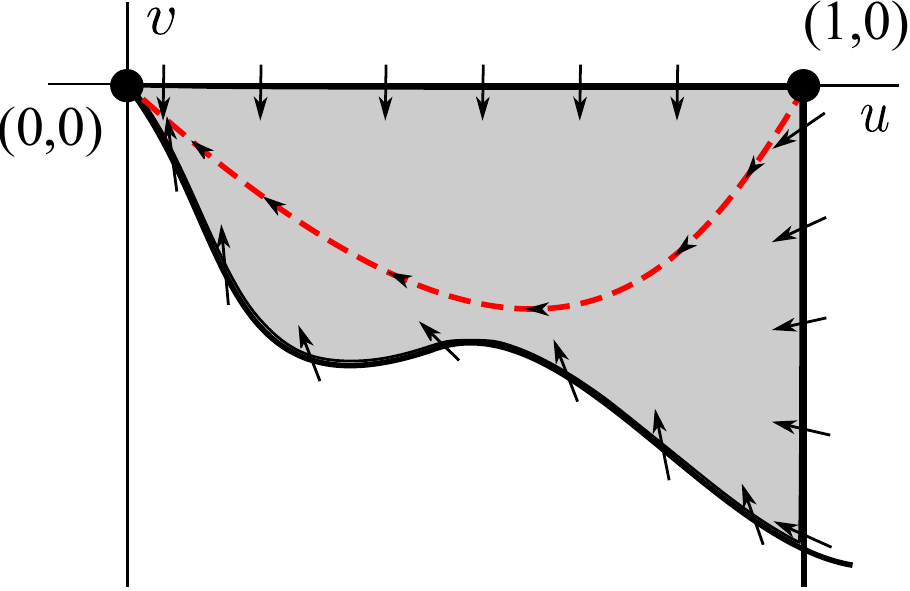} \label{fig:TrapRegion}}
\hfil
\sidesubfloat[]{\includegraphics[width=0.4\textwidth]{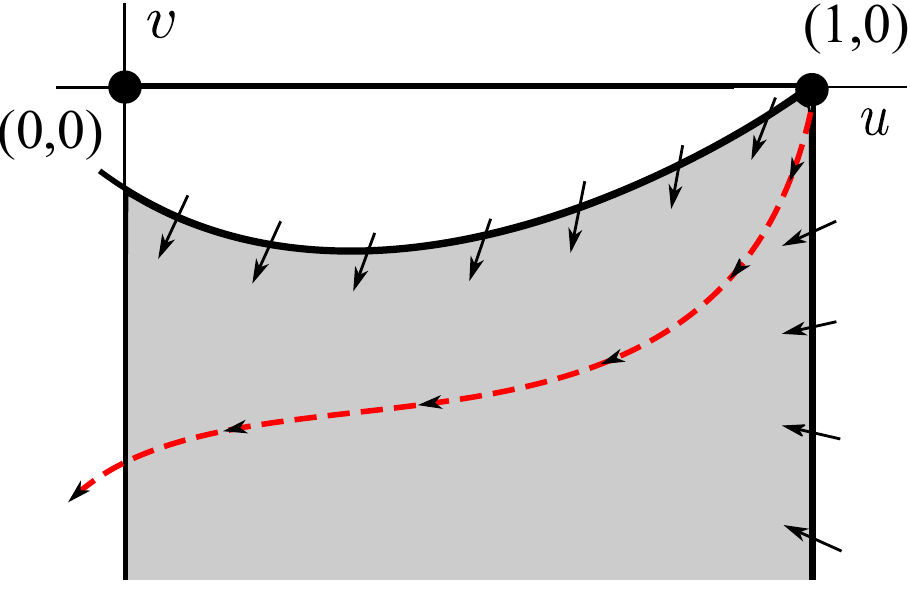} \label{fig:Barrier}}
\caption{(a) An example of a trapping region (shaded) implying that the unstable manifold (${\color{red}\dashedrule}$) of the equilibrium $(1,0)$ of~\cref{eq: scalar dyn syst} connects to $(0,0)$. (b) An example of a barrier that rules out such a connection because trajectories cannot leave the shaded region without leaving $\cS$.}
\label{fig:1DBnds}
\end{figure}


\subsection{\label{subsec:1DUpper}Upper bounds}

To show that the scalar RD equation~\cref{eq: general pde} admits a monotone travelling wave at a given speed $c$, meaning that $c$ is an upper bound on $c_*$, we must verify that the dynamical system~\cref{eq: scalar dyn syst} has a heteroclinic connection from $(1,0)$ to $(0,0)$. \Cref{fig:TrapRegion} qualitatively depicts this connection, along with a trapping region (shaded) that would imply its existence. The top and right boundaries of the shaded region are always trapping as indicated, but it remains to find a bottom boundary. Such a bottom boundary can be sought using the broadly applicable volume method of \Cref{subsec:Volume}. It also can be sought using the surface method of Section~\ref{subsec:Surface}, whose applicability is much narrower but includes systems of the form~\cref{eq: scalar dyn syst} when the desired boundary is a graph that trajectories cross in the positive-$v$ orientation.

To apply the surface method we seek a surface $v=N(u)$ that trajectories within $\cS$ cross in only the positive-$v$ orientation, like the bottom boundary in \cref{fig:TrapRegion}. To verify existence of a heteroclinic connection it suffices to find $N(u)$ such that
\begin{subequations}
\label{eq: N scalar}
\begin{align}
\hspace{-4pt}-y_1^2 \int_0^{u} [D(s)f(s)+(c+a(s))N(s)]\, \drm s + 2y_1y_2N(u) + 2y_2^2 \geq 0 & ~~ \forall~(u,y_1,y_2)\in [0,1]\times \R^2, \label{eq: N scalar 1} \\
-N(u) \ge 0 & ~~ \forall~u\in[0,1], \label{eq: N scalar 2} \\
N(0)=0. \label{eq: N scalar 3}
\end{align}
\end{subequations}
The first constraint is the surface condition~\cref{eq: Nineq}, applied to~\cref{eq: scalar dyn syst} with $\zv=u_1=u$, $\bF_1=0$, $F_2=-D(u)f(u)$, $F_3=-[c+a(u)]$, and $Z=[0,1]$. This implies that trajectories are trapped above the $v=N(u)$ surface, but further constraints on the geometry of this surface are needed to imply a heteroclinic connection. The second and third constraints ensure that the $v=N(u)$ surface passes through the $(0,0)$ equilibrium and lies below the $v=0$ trapping surface, as in \cref{fig:TrapRegion}. This geometry, combined with the monotonicity $\dot u\le 0$ inside $\cS$, implies that the unstable manifold of $(1,0)$ connects to $(0,0)$.

As an alternative to the surface method, the volume method of Section~\ref{subsec:Volume} also can be used to find a trapping boundary like the bottom one in \cref{fig:TrapRegion}. In this case the boundary is defined as the zero level set of a function $V(u,v)$, rather than by $v=N(u)$. It suffices for there to exist $V(u,v)$ and $\lambda,\eps,h>0$ such that
\begin{subequations}
\label{eq: V UB scalar}
\begin{align}
\lambda \left( \big[(c + a(u))v + D(u)f(u) \big]\tfrac{\partial V}{\partial v} - v\tfrac{\partial V}{\partial u}\right) - V(u,v) \ge 0 &\quad \forall~(u,v)\in\cS, \label{eq: V UB scalar 1} \\
-V(u,0) - \eps u (1-u) \ge 0&\quad \forall~u \in [0,1], \label{eq: V UB scalar 2} \\
V(0,v) + \eps v \ge 0&\quad \forall~v \in [-h,0], \label{eq: V UB scalar 3} \\
V(u,-h) \geq 0 &\quad\forall~u \in [0,1], \label{eq: V UB scalar 4} \\
V(0,0) = 0. \label{eq: V UB scalar 5}
\end{align}
\end{subequations}
The first constraint, which is the volume condition~\cref{eq: Trapping} for the dynamical system~\cref{eq: scalar dyn syst}, guarantees that all trajectories staying in $\cS$ are trapped in the negative-$V$ region. The remaining constraints ensure that this region has a geometry as in \cref{fig:TrapRegion}. Because of~\cref{eq: V UB scalar 2,eq: V UB scalar 3,eq: V UB scalar 4}, the rectangle $[0,1]\times[-h,0]$ has $V\le0$ on its top and $V\ge0$ on its left and bottom, and these inequalities are strict on the top and left except at the corners. (For many RD equations the form of~\cref{eq: scalar dyn syst} ensures that $\dot v\ge 0$ on the rectangle's bottom for sufficiently large $h$, in which cases the constraint \cref{eq: V UB scalar 3} can be dropped and the parameter $h$ does not enter.) These facts along with~\cref{eq: V UB scalar 5} mean the $V(u,v)=0$ level set must connect $(0,0)$ to the right boundary of $\cS$. Thus the zero-$V$ curve forms the bottom of a trapping region as in \Cref{fig:TrapRegion}, and the monotonicity $\dot u\le 0$ in this region guarantees a heteroclinic connection from $(1,0)$ to $(0,0)$.

The scalar RD equation~\cref{eq: general pde} is verified to have a travelling wave at a specified speed $c$ if we can find a function $N(u)$ satisfying~\cref{eq: N scalar} or a function $V(u,v)$ and constants $\lambda,\eps,h>0$ satisfying~\cref{eq: V UB scalar}. Assuming all terms in the RD equation are polynomial, we can search for polynomial $N$ or $V$  computationally using SOS methods. This is possible only because $N$ and $V$ appear linearly in their respective constraints. Using the standard approach described in \Cref{subsec:SOS}, in the surface method we replace \cref{eq: N scalar 1,eq: N scalar 2} with sufficient SOS conditions, and in the volume method we replace \crefrange{eq: V UB scalar 1}{eq: V UB scalar 4} with SOS conditions. For each SOS computation in the volume method, one must specify not only the value of $c$ but also $\eps$ and $\lambda$, as well as $h$ in cases where the constraint~\cref{eq: V UB scalar 4} cannot be dropped. It is advantageous to choose $\eps$ as small as possible while remaining larger than numerical error in SOS computations; here we fix $\eps=10^{-4}$. The optimal choice of $\lambda$ depends on the particular RD equation being studied and is explored in the examples below. There is no choice of $h$ in these examples because~\cref{eq: V UB scalar 4} can be dropped. With either the surface or volume method, we seek the smallest possible upper bound on $c_*$ by repeating SOS computations at different $c$ values to find the smallest value at which the existence of a travelling wave can be verified. In~\Cref{subsec: Fisher} we report upper bounds on $c_*$ computed by SOS implementations of both the surface and volume methods. The surface method works at least as well and is simpler to implement, so in \Cref{subsec: chemotaxis} we report upper bounds from the surface method only.


\subsection{Lower bounds}\label{subsec:1DLower}

To show that the scalar RD equation~\cref{eq: general pde} does not admit a monotone travelling wave at a given speed $c$, meaning that $c$ is a lower bound on $c_*$, we must verify that the dynamical system~\cref{eq: scalar dyn syst} cannot have a heteroclinic connection from $(1,0)$ to $(0,0)$. This can be done by finding a trapping boundary like the top of the shaded region in \cref{fig:Barrier}. This boundary, unlike the one in \Cref{fig:TrapRegion}, cannot be found using the surface method since~\cref{eq: scalar dyn syst} has the form~\cref{eq: SurfaceODE} for which a surface $v=N(u)$ may be found that traps trajectories above itself, not below. Thus we use the volume method.

To verify that no heteroclinic connection exists inside $\cS$, it suffices to find $V(u,v)$ and $\lambda,\eps>0$ such that
\begin{subequations}
\label{eq: V LB scalar}
\begin{align}
\lambda \left[ \big([c + a(u)]v + D(u)f(u) \big)\tfrac{\partial V}{\partial v}(u,v) - v\tfrac{\partial V}{\partial u}(u,v)\right] - V(u,v) \ge 0 &\quad \forall~(u,v)\in\cS, \label{eq: V LB scalar 1} \\
V(u,0) - \eps(1-u)\ge0 &\quad \forall~u \in [0,1], \label{eq: V LB scalar 2} \\
V(0,-\eps) = V(1,0) = 0. \label{eq: V LB scalar 3}
\end{align}
\end{subequations}
The first constraint is the same as in~\cref{eq: V UB scalar}, ensuring that trajectories are trapped in the negative-$V$ region. The remaining constraints are different, implying that the $V(u,v)=0$ barrier has a geometry as in \Cref{fig:Barrier} rather than~\ref{fig:TrapRegion}. In particular, the barrier passes below $(0,0)$ and through $(1,0)$ without leaving $\cS$, and $(0,0)$ lies on the positive-$V$ side of the barrier. To use computational SOS methods to find polynomial $V(u,v)$ satisfying~\cref{eq: V LB scalar}, we replace \cref{eq: V LB scalar 1,eq: V LB scalar 2} with sufficient SOS conditions by the approach described in \cref{sec:Methods}\ref{subsec:SOS}. To find the largest possible lower bound on $c_*$, we repeat the SOS computations at different $c$ values to find the largest value at which a travelling wave can be ruled out. Results of such computations giving lower bounds on $c_*$ for a generalized Fisher--KPP equation and a chemotaxis model are reported in the next two subsections.


\subsection{\label{subsec: Fisher}Application to a generalized Fisher--KPP equation}

To demonstrate the application of our methods to scalar RD equations of the form~\cref{eq: general pde}, we first consider a generalized Fisher--KPP equation~\cite{Billingham,Fisher,KPP},
\begin{equation}
\label{Fisher}
	u_t = u_{xx} + u^m(1-u),
\end{equation}
with parameter $m\ge1$. The well understood $m=1$ case is a classic example in the study of travelling wave solutions to RD equations, and generalizations to $m>1$ have been considered to model order-$m$ autocatalytic reactions~\cite{Billingham} and weak Allee effects in spatial ecology~\cite{Courchamp}. The dynamical system~\cref{eq: scalar dyn syst} governing travelling waves in this case is
\begin{equation}
\label{mFisherODE}
	\begin{split}
		\dot{u} &= v, \\
		\dot{v} &= -cv - u^m(1-u).
	\end{split}
\end{equation} 
Since the present example has $D(u)\equiv1$ in~\cref{eq: general pde}, the independent variable is simply $\xi = x - ct$, and the phase vector is $(u,u_\xi)$.

Travelling waves solutions of~\cref{Fisher} exist if and only if their speed exceeds some minimum speed $c_*$~\cite{Billingham}, which depends on $m$, but exact value of $c_*$ are known only when $m=1$ or $m=2$. Analytical upper bounds on $c_*$ have been derived in the past using linear trapping boundaries in the $(u,u_\xi)$-plane, which amounts to our surface method with $v=N(u)$ being simply a line. In the classic $m=1$ case, the exact value $c_*=2$ is known~\cite{Fisher,KPP,Billingham} because a linear trapping boundary verifies existence of a monotone travelling wave when $c\geq2$, and linearization of~\cref{mFisherODE} shows that all trajectories approaching $(0,0)$ as $\xi \to \infty$ must leave the region $\cS$ defined by~\cref{eq: S} for any $c<2$. In the $m=2$ case the sharp value $c_*=1/\sqrt{2}$ is known also, with trapping region arguments using the stable manifold of the origin to give sharp upper and lower bounds~\cite{Billingham2}. For general $m>1$, using linear trapping boundaries to verify existence of monotone travelling waves gives the upper bound $c_* \leq 2m^{-m/2}(m-1)^{(m-1)/2}$ \cite{Britton, Billingham}. These upper bounds cannot be sharp since strictly slower waves emerge when the governing RD equations are numerically integrated~\cite{Billingham}. Nonzero lower bounds have not been reported; linearization of~\cref{mFisherODE} about $(0,0)$ cannot rule out monotone travelling waves at any positive $c$ because the point is nonhyperbolic when $m>1$. \Cref{lem:FisherBnd} below gives a sharper upper bound on $c_*$, derived analytically by optimizing the trapping surface $v=N(u)$ among linear $N$. Then we turn to SOS computations for sharper upper and lower bounds, at least for integer values of $m$.

\begin{lem}
\label{lem:FisherBnd}
	For every $m > 1$, the minimum speed of monotone travelling waves in~\cref{Fisher} asymptotically connecting $u(-\infty) = 1$ and $u(\infty) = 0$ is bounded above according to
	\begin{equation}
	\label{lemFisher}
		c_* \leq 2\sqrt{\frac{2[(m-1)(m+2)]^{m-1}}{[m(m+1)]^m}}.
	\end{equation}
\end{lem}

\begin{proof}
We show existence of a heteroclinic connection in~\cref{mFisherODE} using the surface method with linear $N(u)$. Let $N = -\eta u$ with $\eta > 0$ to be chosen. This $N$ satisfies the latter two constraints in~\cref{eq: N scalar}, while the first constraint requires
	\begin{equation}
		\frac{1}{m+2}u^{m+2} - \frac{1}{m+1}u^{m+1} + \frac{c \eta}{2}u^2 - \frac{\eta^2}{2}u^2 \geq 0	
	\end{equation}
for all $u \in [0,1]$. Factoring out $u^2$ gives the equivalent condition 
\begin{equation}
	\label{eq: eta cond}
	\frac{1}{m+1}u^{m-1} - \frac{1}{m+2}u^m \le \frac{c \eta}{2}- \frac{\eta^2}{2}.
\end{equation}
The maximum of the left-hand expression over $u\in[0,1]$ is $k(m):=\frac{[(m-1)(m+2)]^{m-1}}{[m(m+1)]^m}$. We choose $\eta=c/2$ so that the right-hand expression takes its largest possible value of $c^2/8$. The inequality~\cref{eq: eta cond} then holds for all $u\in[0,1]$ if and only if $c\ge \sqrt{8k(m)}$. The minimum such $c$ provides the upper bound~\cref{lemFisher}.
\end{proof} 

We have computed upper and lower bounds on $c_*$ by the SOS methods described in \Cref{subsec:1DUpper,subsec:1DLower}. To carry out such computations we choose a maximum polynomial degree for the auxiliary function $N(u)$ or $V(u,v)$. We also choose degrees for the additional tunable polynomials that are introduced when nonnegativity on a semialgebraic set is enforced using SOS conditions, as explained in \Cref{subsec:SOS} and exemplified by $\sigma_i$ in~\cref{eq: p sos}. In all computations we report for the RD equation~\cref{Fisher}, the degree of these additional polynomials was $m$ or $m-1$ larger than the degree of the auxiliary function, depending on which gives an even degree. For various degrees of the auxiliary functions, and for various $\lambda$ values in the case of the volume method, we have searched over $c$ for the best possible upper and lower bounds on $c_*$. 

\begin{table}[t] 
\centering
\caption{\label{FisherTable}Upper and lower bounds on the minimum speed of monotone travelling waves solutions of~\cref{Fisher} for various $m$. Upper bounds were computed using the surface method with $\deg(N)=~20$. Lower bounds were computed using the volume method with $\deg(V)=~20$ and $(\eps,\lambda) = (10^{-4},10^3)$. Also shown is the measured speed of emergent travelling waves in numerical integrations of the PDE.}
 \begin{tabular}{l c c c c c c c c c} 
 \hline
 $m$ & 2 & 3 & 4 & 5& 6 & 7 & 8 & 9 & 10 \\
 \hline
Upper bound & 0.7071 & 0.4632 & 0.3467 & 0.2776 & 0.2317 & 0.1989 & 0.1742 & 0.1550 & 0.1397 \\
PDE integration & 0.7071 & 0.4632 & 0.3467 & 0.2776 & 0.2316 & 0.1988 & 0.1742 & 0.1550 & 0.1397 \\
Lower bound & 0.7068 & 0.4629 & 0.3465 & 0.2774 & 0.2315 & 0.1987 & 0.1741 & 0.1549 & 0.1395 \\
\end{tabular}
\end{table} 

Auxiliary functions of sufficiently high degree give upper and lower bounds on $c_*$ that are very close to sharp. \cref{FisherTable} illustrates this by showing, for various $m$ values, upper bounds computed using the surface method with $\deg(N)=20$ and lower bounds computed using the volume method with $\deg(V)=20$. Also shown in \cref{FisherTable} are the post-transient speeds of travelling waves that emerged when the RD equation~\cref{Fisher} was numerically integrated using a second-order finite difference scheme with a Heaviside function as the initial condition. In the $m = 1$ case  it has been proved that this initial condition converges to the minimum-speed travelling wave~\cite{KPP}. Although the same result has not been proved for $m>1$, \cref{FisherTable} gives very strong evidence for it. Each pair of upper and lower bounds in \cref{FisherTable} agrees to 3 significant digits and is consistent with the $c_*$ suggested by numerical integration. We now turn to how such bounds on $c_*$ depend on the polynomial degrees of auxiliary functions and the value of $\lambda$.

\floatsetup[figure]{style=plain,subcapbesideposition=top}	
\begin{figure} 
\centering
\sidesubfloat[]{\includegraphics[trim={15pt 0pt 20pt 15pt},clip,width=0.45\textwidth]{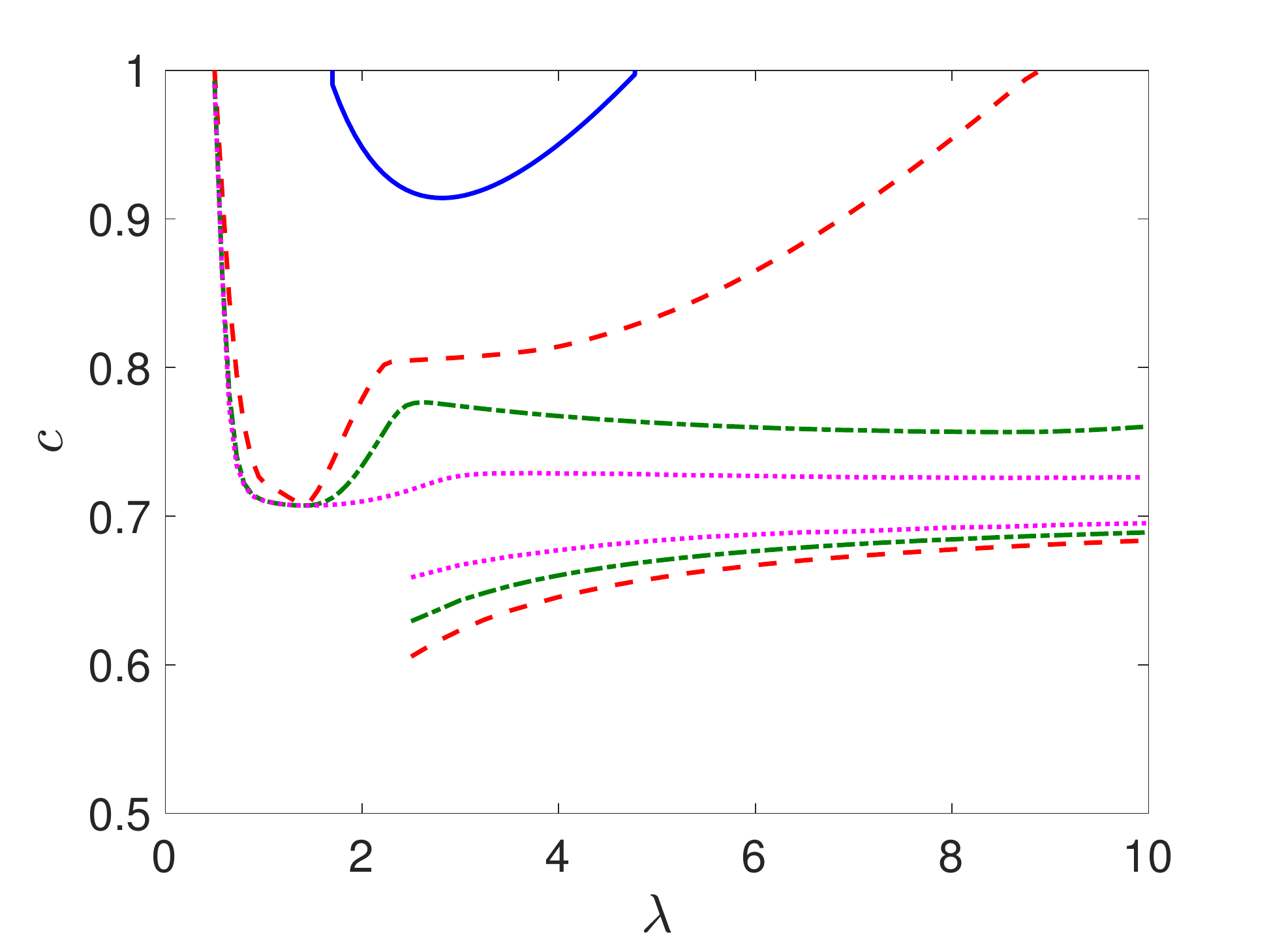}\label{fig:FisherLambda_a}}
\hfil
\sidesubfloat[]{\includegraphics[trim={15pt 0pt 20pt 15pt},clip,width=0.45\textwidth]{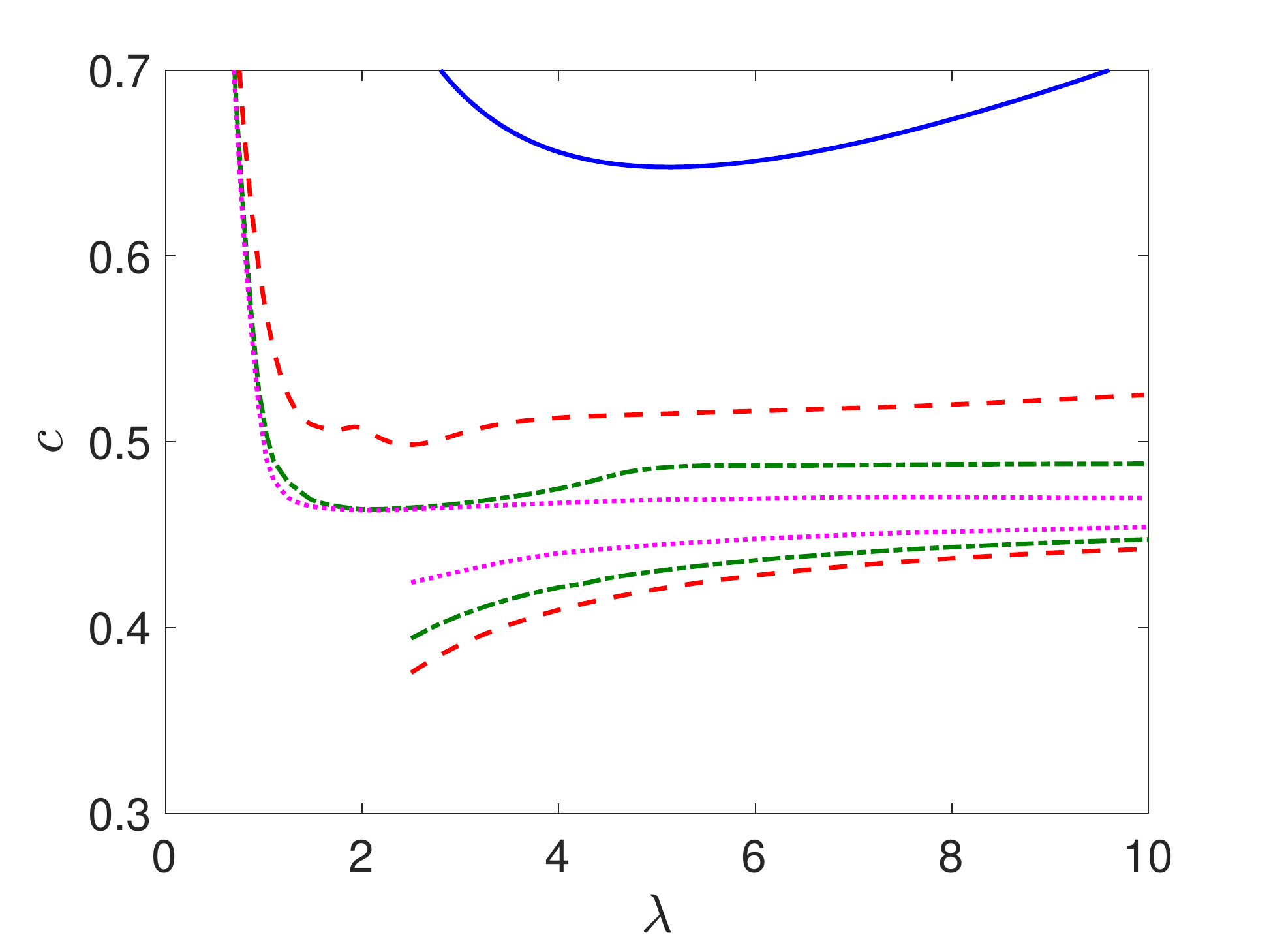}\label{fig:FisherLambda_b}}
\caption{Dependence on $\lambda$ of bounds on $c_*$ for the generalized Fisher--KPP equation \cref{Fisher} with (a) $m=2$ and (b) $m=3$, found by SOS computations using the volume method. Each pair of lines denotes upper and lower bounds computed with $\eps=10^{-4}$ and $\deg(V)=2$ ({\color{blue}$\solidrule$}), 3 ({\color{red}$\dashedrule$}), 4 ({\color{mygreen}$\dashdottedrule$}), and 5 ({\color{mypink}$\dottedrule$}). Lower bounds are not shown for $\lambda<2.5$ due to numerical inaccuracy. With $\deg(V)=2$, nonzero lower bounds were not found.}
\label{fig:FisherLambda}
\end{figure}

We have used the volume method to find both upper and lower bounds by computing polynomial $V(u,v)$ that satisfy the SOS relaxations of~\cref{eq: V UB scalar} and~\cref{eq: V LB scalar}, respectively. In the upper bound formulation the condition~\cref{eq: V UB scalar 4} is unneeded for the present example of~\cref{mFisherODE}, as explained after~\cref{eq: V UB scalar}. We fix $\eps=10^{-4}$ throughout; changing this value has only an order-$\eps$ effect on the optimal bound. To examine the effect of $\lambda$ in the $m=2,3$ cases we have computed the best possible upper and lower bounds over a range of $\lambda$ using $\deg(V)=2,3,4,5$. \Cref{fig:FisherLambda} shows the results. Upper bounds are minimized by intermediate values of $\lambda$, whereas lower bounds appear to asymptote to their suprema as $\lambda\to\infty$. We do not know how to anticipate $\lambda$-dependance in general, but bounds become insensitive to $\lambda$ as $\deg(V)$ increases, at least for $\lambda\gtrsim1$. In this example it therefore works well in practice to fix a sufficiently large value of $\lambda$ and carry out SOS computations with $V$ of successively higher degree until bounds become sharp. As an example we have computed the best possible upper and lower bounds with the fixed values $\lambda=3$ and $\lambda=10$, respectively, using $V$ of increasing degree. \cref{FisherTable2} summarizes the results for the $m=2,3,4,5$ cases, showing that upper and lower bounds converge quickly towards each other as $\deg(V)$ is raised.

\begin{table}[t] 
\centering
 \begin{tabular}{l c c c c c c c c} 
 \hline
 Degree & 1& 2 & 3 & 4 & 5 & 6 & 7 & 8 \\ 
 \hline
 Upper bound (vol.) & 1.0832 & 0.9150 & 0.8068 & 0.7740 & 0.7269 & 0.7161 & 0.7106 & 0.7097 \\ 
Upper bound (surf.) & 0.9428 & 0.7071 & 0.7071 & 0.7071 & 0.7071 & 0.7071 & 0.7071 & 0.7071 \\ 
Lower bound (vol.) & --- & --- & 0.6836 & 0.6891 & 0.6952 & 0.6957 & 0.6964 & 0.6968 
 \end{tabular}
 \begin{tabular}{l c c c c c c c c} 
 \hline
 Degree & 1& 2 & 3 & 4 & 5 & 6 & 7 & 8\\ 
 \hline
Upper bound (vol.)& 0.7698 & 0.6879 & 0.5035 & 0.4665 & 0.4646 & 0.4632 & 0.4628 & 0.4627\\ 
Upper bound (surf.) & 0.6804 & 0.5102 & 0.4667 & 0.4637 & 0.4632 & 0.4632 & 0.4632 & 0.4632\\ 
Lower bound (vol.) & --- & --- & 0.4423 & 0.4475 & 0.4541 & 0.4526 & 0.4558 & 0.4561\\ 
 \end{tabular}
 \begin{tabular}{l c c c c c c c c} 
 \hline
 Degree & 1& 2 & 3 & 4 & 5 & 6 & 7 & 8\\ 
 \hline
 Upper bound (vol.) & 0.6495 & 0.6021 & 0.3973 & 0.3477 & 0.3466 & 0.3465 & 0.3465 & 0.3465\\ 
Upper bound (surf.) & 0.5400 & 0.4048 & 0.3493 & 0.3474 & 0.3469 & 0.3467 & 0.3467 & 0.3467\\ 
Lower bound (vol.) & --- & --- & 0.3162 & 0.3276 & 0.3375 & 0.3382 & 0.3408 & 0.3408\\ 
\end{tabular}
 \begin{tabular}{l c c c c c c c c} 
 \hline
 Degree & 1& 2 & 3 & 4 & 5 & 6 & 7 & 8 \\ 
 \hline
 Upper bound (vol.)& 0.5724 & 0.5550 & 0.3348 & 0.2844 & 0.2775 & 0.2775 & 0.2775 & 0.2775\\ 
Upper bound (surf.) & 0.4498 & 0.3372 & 0.2853 & 0.2779 & 0.2776 & 0.2776 & 0.2776 & 0.2776\\ 
Lower bound (vol.) & --- & --- & 0.2356 & 0.2523 & 0.2683 & 0.2695 & 0.2720 & 0.2720\\ 
 \end{tabular}
 \caption{\label{FisherTable2}Numerically computed upper and lower bounds on the minimum speed $c_*$ of monotone travelling waves of~\cref{Fisher} for $m = 2,3,4,5$ (from top to bottom). Degrees of $N$ in the surface method and of $V$ in the volume method range from 1 to 8. In the volume method $\eps=10^{-4}$, and $\lambda=3$ or 10 for upper or lower bounds, respectively. Lower bounds are not reported for $\deg(V)\le2$ because computations do not verify nonexistence of travelling waves for any positive $c$ values. Results of the surface method with $\deg(N)=1$ agree with the analytical expression \cref{lemFisher}. Surface method bounds are sharp to the tabulated precision when $\deg(V)\ge6$ (cf.\ \cref{FisherTable}).}
\end{table}  

We computed upper bounds on $c_*$ for~\cref{Fisher} using the surface method as well, computing polynomial $N(u)$ that satisfy the SOS relaxation of~\cref{eq: N scalar}. \Cref{FisherTable2} shows the best upper bounds computed in the $m=2,3,4,5$ cases using $N$ of various polynomial degrees. These bounds become sharp as $\deg(N)$ is raised, and the convergence is faster than the convergence of the volume method when $\deg(V)$ is raised with the suboptimal value $\lambda=3$ fixed. This is an apparent advantage of the surface method for computing upper bounds. The surface method also is simpler to implement because the volume method requires the values of $\varepsilon$ and $\lambda$ to be either chosen \emph{a priori} or optimized via repeated SOS computations. We therefore use only the surface method to compute upper bounds in the next example, although the volume method still is needed to compute lower bounds.


\subsection{\label{subsec: chemotaxis}Application to a chemotaxis model}

To further illustrate the success of our methods for finding minimum wave speeds in scalar RD equations, we consider a model from~\cite{Mansour} for chemotaxis of bacteria in a one-dimensional strip. The density $u(x,t)$ of bacteria is modelled by
\begin{equation}
\label{Chemo}
	u_t = (u^ku_x)_x - b u u_x + u(1-u^q)
\end{equation}
with parameters $k,q > 0$ and $b \geq 0$. It was shown in~\cite{Mansour} that there is a finite minimum speed $c_*$ above which monotone travelling waves exist. The exact value of $c_*$, which depends on the parameters, was not found. Here we produce upper and lower bounds on $c_*$ using SOS computations, and we derive upper bounds analytically.

In the case of \cref{Chemo}, the dynamical system~\cref{eq: scalar dyn syst} governing travelling waves takes the form
\begin{equation}
\label{mChemoODE}
	\begin{split}
		\dot{u} &= v, \\
		\dot{v} &= (b u - c)v - u^{k+1}(1-u^q).
	\end{split}
\end{equation}
Analytical estimates of $c_*$ are available only for certain values of the parameters. When $(q,b)=(1,0)$, the value of $c_*$ for the chemotaxis equation~\cref{Chemo} is the same as for the Fisher--KPP equation~\cref{Fisher} with $m=k+1$, as follows from the equivalence of the dynamical systems~\cref{mFisherODE} and~\cref{mChemoODE}. Here we provide an analytical upper bound on $c_*$ for all $(k,q,b)$, although it is not generally sharp. As done for the Fisher--KPP equation in \cref{lem:FisherBnd} above, we derive the bound using the surface method with a linear trapping boundary $v=N(u)$. \Cref{lem:ChemoBnd} below states the result, whose proof we omit because it is analogous to the proof of~\cref{lem:FisherBnd}.

\begin{lem}
\label{lem:ChemoBnd}
	For every $k,q > 0$ and $b \geq 0$, the minimum speed of monotone travelling waves in \cref{Chemo} asymptotically connecting $u(-\infty) = 1$ to $u(\infty) = 0$ is bounded above according to
	\begin{equation}
		c_* \leq \frac{2b}{3} + 2\sqrt{\frac{2qk^\frac{k}{q}(k+q+2)^\frac{k}{q}}{(k + 2)^{1 + \frac{k}{q}}(k+q)^{1 + \frac{k}{q}}}}.
	\end{equation}
\end{lem} 

Minimum wave speeds of~\cref{Chemo} are unknown for most parameter values. As a test of our methods, we have used SOS computations to find bounds on $c_*$ with $(k,b)=(2,1)$ and integer $q$ values from 1 to 9. We use the surface method with $\deg(N)=20$ for upper bounds and the volume method with $\deg(V)=20$ for lower bounds. The additional polynomials introduced to enforce nonnegativity on semialgebraic sets via SOS conditions---i.e., the $\sigma_i$ in~\cref{eq: p sos} and their analogues---all had degree $20+q$ or $20+q-1$ depending on whichever is even. For the volume method we fixed $\eps=10^{-4}$ and $\lambda=10^3$. \Cref{ChemoTable} reports the resulting bounds. All upper and lower bounds agree to three digits. In the $q=1$ case, the value $c_*\approx0.8239$ was estimated in~\cite{Mansour} by finding heteroclinic connections in~\cref{mChemoODE} using a shooting method, and our upper bound agrees with this value to all 4 digits.

\begin{table}[t] 
\centering
 \begin{tabular}{l c c c c c c c c c c} 
 \hline
 $q$ & 1 & 2 & 3 & 4 & 5& 6 & 7 & 8 & 9 \\
 \hline
Upper bound & 0.8239 & 0.9433 & 1.0120 & 1.0557 & 1.0967 & 1.1106 & 1.1289 & 1.1435 & 1.1554 \\ 
Lower bound & 0.8236 & 0.9429 & 1.0104 & 1.0546 & 1.0856 & 1.1095 & 1.1274 & 1.1422 & 1.1550 \\ 
 \end{tabular}
  \caption{Upper and lower bounds on the minimum speed of monotone travelling waves of (\ref{Chemo}) for $(k,b) = (2,1)$ and various $q$, found using SOS computations with auxiliary functions of polynomial degree 20 (see text).}
   \label{ChemoTable}
\end{table}


\section{\label{sec:Multi}Application to a two-component cubic autocatalysis model}

In order to verify the existence or nonexistence of travelling waves in multicomponent RD equations using the methods of \Cref{sec:Methods}, one must confront difficulties that are absent in the case of scalar RD equations. The dynamical system~\cref{eq: dyn syst} governing travelling waves will have a dimension of at least three. Because the Poincar\'e--Bendixson theorem no longer applies, heteroclinic connections cannot be verified solely by finding trapping boundaries and ruling out periodic orbits; additional arguments based on monotonicity are needed. Whereas the case of scalar RD equations was handled in generality in \Cref{sec:Methods}, we do not give a general treatment of the multicomponent case. Instead we illustrate the application of our methods to an example of a two-component RD equation from the recent literature. 

We consider an autocatalytic chemical reaction model that has received considerable attention~\cite{Billingham1991a, Billingham1991b, Chen, Chen2, Focant, Hosono, Qi},
\begin{equation}
\label{Chemical}
	\begin{split}
		\alpha_t &= \alpha_{xx} - \alpha\beta^m, \\
		\beta_t &= D\beta_{xx} + \alpha\beta^m,
	\end{split}
\end{equation}
with parameters $m\geq 1$ and $D > 0$. Most studies have focused on quadratic ($m=1$) or cubic ($m=2$) autocatalysis. Here we report computations for cubic autocatalysis, which is the less understood of the two cases. Since $\alpha(x,t)$ and $\beta(x,t)$ model chemical concentrations, we restrict attention to travelling waves where these functions are nonnegative, in particular waves where $(\alpha,\beta) \to (1,0)$ as $\xi \to -\infty$ and $(\alpha,\beta) \to (0,1)$ as $\xi\to \infty$. It has been proved for all $D>0$ and $m\ge1$ that such waves exist if and only if their speeds exceed some minimum $c_*$~\cite{Chen}, but exact values of $c_*$ are known only when $D=1$ or $m=1$. When $D = 1$ the existence of nonnegative travelling waves in~\cref{Chemical} is equivalent to that of nonnegative travelling waves in the generalized Fisher--KPP equation~\cref{Fisher}, as detailed in the introduction of~\cite{Billingham}.

Bounds on $c_*$ have been derived analytically which depend on $D>0$ and $m>1$~\cite{Billingham1991a,Chen2,Focant,Qi}, but there is a gap between the best upper and lower bounds when $D\neq1$. In the $m=2$ case on which we focus, the best upper and lower bounds that have been proved analytically are
\begin{equation}
\label{eq: m=2 analyt}
\begin{aligned}
\underline{D<1}:\quad\frac{D}{\sqrt2} \le c_* &\le \min\bigg\{\frac{4D}{\sqrt{1 + 4D}},\sqrt{D}\bigg\}, & 
	\hspace{20pt}\underline{D>1}:\quad\sqrt{\frac{D}{2}} \le c_* &\le \sqrt{\frac{D}{1+1/D}}.
\end{aligned}
\end{equation}
All of these bounds are derived in~\cite{Chen2}, except for the upper bound $c_*\le\sqrt{D}$ from~\cite{Qi} which is the best available when $1/12<D<1$. Here we report sharper numerical bounds for many values of~$D$.

\subsection{\label{sec: autocat method}Formulations for computing bounds}

The existence of a travelling wave of speed $c$ in~\cref{Chemical} is equivalent the existence of a heteroclinic connection from $(u,v,w)=(0,0,0)$ to $(1,1,0)$ in the dynamical system
\begin{equation}
\label{ChemicalODE}
	\begin{split}
		\dot{u} &= D(v + w - u), \\
		\dot{v} &= w, \\
		\dot{w} &= -w + \frac{D}{c^2}u(1 - v)^m,
	\end{split}	
\end{equation}
where a dot denotes $\tfrac{\drm}{\drm\xi}$. See \cite{Chen} for the demonstration of this equivalence and the definitions of $u,v,w$, which make use of a conserved quantity. In particular, it has been shown \cite{Billingham1991a} that such a heteroclinic connection corresponding to a nonnegative travelling wave of~\cref{Chemical} must lie in the $D$-dependent region
\begin{equation} \label{eq: P1}
	\cP = \begin{cases} 
			\{(u,v,w):\ 0 \leq u \leq v \leq 1,\ w \geq 0\}, & D<1, \\
			\{(u,v,w):\ 0 \leq v \leq u \leq 1,\ w \geq 0\}, & D>1.
		\end{cases}
\end{equation}
If the heteroclinic connection exists it is part of the unstable manifold of the origin, which is partly characterized by the following lemma. We omit details of the straightforward proof: part (i) follows from the linearization of~\cref{ChemicalODE} at the origin, and part (ii) follows from the observation that all boundaries of $\cP$ aside from one are trapping.

\begin{lem}\label{lem:Chemical} 
In the dynamical system~\cref{ChemicalODE} with any $D\neq1$, $c > 0$, and $m \geq 1$, the equilibrium at the origin (i) has a one-dimensional unstable manifold that enters $\cP$ and (ii) can leave $\cP$ only by crossing the $v = 1$ boundary when $D < 1$ and the $u = 1$ boundary when $D > 1$.
\end{lem} 

To show that some wave speed $c$ is an upper bound on $c_*$, we must verify that the unstable manifold of $(0,0,0)$ connects to $(1,1,0)$ inside $\cP$. Since~\cref{ChemicalODE} does not have the particular form~\cref{eq: SurfaceODE} required for the surface method, we use the volume method. (Alternatively, as described after~\cref{eq: SurfaceODE}, one could apply the surface method to a four-dimensional ODE governing travelling waves of~\eqref{Chemical}.) It suffices to find $V(u,v,w)$ and $\lambda,\eps>0$ satisfying
\begin{subequations}
\label{eq: upper autocat}
\begin{align}
-\lambda \bF(u,v,w)\cdot\nabla V(u,v,w) - V(u,v,w) \ge 0 &\quad \forall~(u,v,w)\in\cP, \label{eq: upper autocat 1} \\
-V(0,0,0) - \eps \geq 0, \label{eq: upper autocat 3} \\
V(1,1,0) = 0, \label{eq: upper autocat 4}
\end{align}
\end{subequations}
and
\begin{equation}
\label{eq: upper autocat 2}
\begin{array}{ll}
\underline{D<1}:\quad V(u,1,w) - \eps w \ge0 &\forall~(u,\cdot,w) \in \cP, \\ 
\underline{D>1}:\quad V(1,v,w) - \eps w \ge0 &\forall~(\cdot,v,w) \in \cP,
\end{array}
\end{equation}
where $\bF(u,v,w)$ is the right-hand side of the dynamical system~\cref{ChemicalODE}. The first constraint is simply the volume condition~\cref{eq: Trapping} for this dynamical system. It guarantees that all trajectories staying in $\cP$ are trapped in the negative-$V$ region. The second constraint ensures that the origin is inside this region, while the third puts the target equilibrium on the boundary. In the $D<1$ case, \cref{eq: upper autocat 2} implies that the $v = 1$ boundary of $\cP$, aside from the edge where $w=0$, is in the positive-$V$ region and so cannot be reached by the unstable manifold of the origin. In the $D > 1$ case, \cref{eq: upper autocat 2} implies the same for the $u=1$ boundary of $\cP$. It was proven in~\cite{Chen} that the unstable manifold of the origin can only have $w = 0$ in $\cP$ at $\xi = \pm\infty$, implying that if $V$ can be constructed to the above specifications, the unstable manifold cannot leave $\cP$ through $w = 0$. Hence, a function $V$ satisfying~\eqref{eq: upper autocat} guarantees that the unstable manifold of the origin must arrive at the equilibrium $(1,1,0)$ as $\xi \to \infty$. 

To show that some wave speed $c$ is a lower bound on $c_*$, we must verify that the unstable manifold of $(0,0,0)$ cannot connect to $(1,1,0)$ inside $\cP$. We use the volume method to seek a function $V(u,v,w)$ whose zero level set forms a suitable barrier. It suffices to find $V(u,v,w)$ and $\lambda,\eps>0$ such that
\begin{subequations}
\label{eq: lower autocat}
\begin{align}
-\lambda \bF(u,v,w)\cdot\nabla V(u,v,w) - V(u,v) \ge 0 &\quad \forall~(u,v,w)\in\cP, \label{eq: lower autocat 1} \\
V(u,v,0) - \eps(u + v) \ge0 &\quad \forall~(u,v,\cdot) \in \cP, \label{eq: lower autocat 2} \\
V(0,0,0) = 0, \label{eq: lower autocat 3}
\end{align}
\end{subequations}
where $\bF(u,v,w)$ is the right-hand side of the dynamical system~\cref{ChemicalODE}. The first constraint again guarantees that all trajectories staying in $\cP$ are trapped in the negative-$V$ region. The third constraint ensures that the origin is on the boundary of the trapping region. One half of the unstable manifold of the origin must enter the region where $V\le 0$ because otherwise, since the origin is a saddle, not all nearby trajectories could be trapped in that region. The second constraint in~\cref{eq: lower autocat} ensures that the equilibrium $(1,1,0)$ is outside of the trapping region, so it cannot be reached by the part of the origin's unstable manifold that enters $\cP$.

\subsection{\label{sec: autocat results}Computed bounds}

We have computed upper and lower bounds on $c_*$ using polynomial $V(u,v,w)$, replacing the nonnegativity constraints in~\cref{eq: upper autocat} and~\cref{eq: lower autocat} with SOS sufficient conditions as descried in \cref{sec:Methods}\ref{subsec:SOS}. Note that $\cP$ can be defined in the form~\cref{eq: U} of a semialgebraic set. In the $D<1$ case, say, this can be done by letting $s_1=u$, $s_2=v-u$, $s_3=1-v$, and $s_3=w$. This introduces additional tunable polynomials, such as the $\sigma_i$ in~\cref{eq: p sos}, whose degrees we fix to be the same as the degree of $V$. Fixing the degree of $V$, we find that the dependence of bounds on $\lambda$ is much as in the generalized Fisher--KPP example of \cref{fig:FisherLambda}: upper bounds are minimized at an order-one value of $\lambda$, while lower bounds are maximized as $\lambda\to\infty$.

\begin{figure}[t]
\centering
\includegraphics[trim={60pt 2pt 60pt 6pt},clip,width=.9\textwidth]{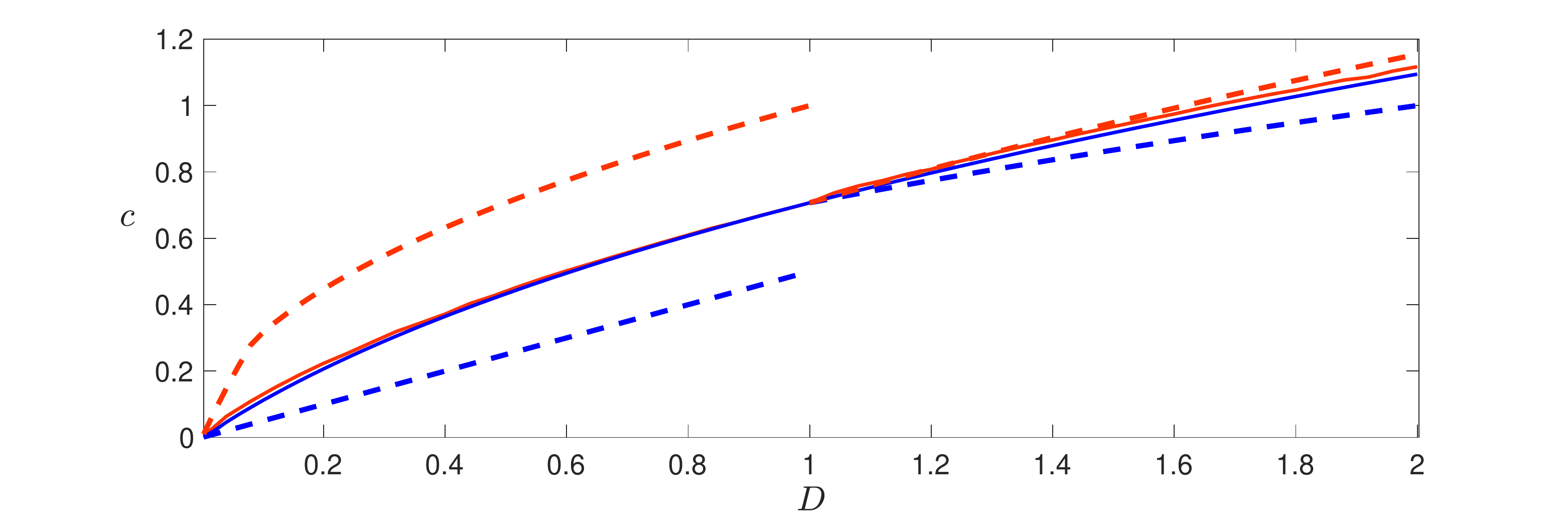} 
\caption{Upper (${\color{red}\blacksquare}$) and lower (${\color{blue}\blacksquare}$) bounds on $c_*$ for~\cref{Chemical}, from our SOS computations ($\solidrule$) and the analytical estimates~\cref{eq: m=2 analyt} ($\dashedrule$). Computations used $\deg(V)=6$ with $(\varepsilon,\lambda) = (10^{-4},0.5)$ for upper bounds and $(\varepsilon,\lambda) = (10^{-4},10^3)$ for lower bounds.}
\label{fig:ChemicalBnds}
\end{figure} 

\Cref{fig:ChemicalBnds} shows numerical upper and lower bounds on $c_*$ computed using the volume method via \cref{eq: upper autocat} and \cref{eq: lower autocat}, respectively, for $D\in(0,2]$. Also shown are the best bounds~\cref{eq: m=2 analyt} that have been proved analytically. The plotted numerical bounds were computed using $\deg(V) = 6$ and $\varepsilon = 10^{-4}$ with $\lambda=0.5$ for upper bounds and $\lambda=10^3$ for lower bounds. The upper and lower bounds from SOS computations nearly coincide in the figure, implying that both are very close to being sharp. They also show that the analytical bounds are not sharp, except when $D=0$ or $D=1$.

The exact dependence of $c_*$ on $D$ has not been found analytically, but its asymptotic scaling is $c_*\sim C_0D$ for $D\ll1$ and $c_*=C_\infty\sqrt{D}$ for $D\gg1$. These scalings were first noted in travelling waves that emerged when the RD system~\cref{Chemical} was numerically integrated, and they are guaranteed by the analytical bounds~\cref{eq: m=2 analyt}. The asymptotic constants $C_0$ and $C_\infty$ are not known exactly, although it follows from~\cref{eq: m=2 analyt} that $1/\sqrt2\le C_0\le4$ and $1/\sqrt2\le C_\infty\le 1$, and $C_0\approx1.219$ was suggested in~\cite{Billingham1991a} based on their numerical integration. In order to estimate $C_0$ and $C_\infty$ here, we have computed upper and lower bounds on $c_*$ that are even sharper than those shown in \Cref{fig:ChemicalBnds}. We did this by raising $\deg(V)$ up to 14 and optimizing over $\lambda$. In the $D\ll1$ regime where $c_*$ is small, we want to choose $\eps$ as small as possible because computed bounds are conservative by at least an order-$\eps$ margin. We cannot take $\eps$ smaller than $10^{-5}$ due to numerical ill conditioning, however, so our bounds cease to be sharp for very small $D$. We nonetheless can probe the asymptotic scaling of $c_*$ because our upper and lower bounds agree to within three significant digits for $D$ as small as 0.005 and as large as $10^4$. \Cref{tab: asymptotic D} summarizes the results, which suggest that $C_0\approx1.2$ and $C_\infty\approx0.861$.

\begin{table}
 \begin{tabular}{l c c c c c c c c} 
 \hline
 D & $10^{-3/4}$ & $10^{-1}$ & $10^{-5/4}$ & $10^{-3/2}$ & $10^{-7/4}$ & $10^{-2}$ & $10^{-9/4}$     \\
 \hline
$c_*$ & 0.185 & 0.115 & 0.0650 & 0.0373 & 0.0212 & 0.0120 & 0.00675    \\
$c_*/D$ & 1.04 & 1.15 & 1.16 & 1.18 & 1.19 & 1.20 & 1.20
 \end{tabular} \vskip 0.3cm
 \begin{tabular}{l c c c c c c c c} 
 \hline
 D & $10^1$ & $10^{3/2}$ & $10^2$ & $10^{5/2}$ & $10^3$ & $10^{7/2}$ & $10^4$ & \\
 \hline
$c_*$ & 2.66 & 4.80 & 8.59 & 15.3 & 27.2 & 48.4 & 86.1 \\
$c_*/\sqrt{D}$ & 0.842 & 0.855 & 0.859 & 0.861 & 0.861 & 0.861 & 0.861
 \end{tabular}
  \caption{Approximations of the minimum wave speeds $c_*$ for small $D$ (top) and large $D$ (bottom), along with the ratios that approximate prefactors of asymptotic scalings when $D\ll1$ and $D\gg1$. Values of $c_*$ reflect computed upper and lower bounds that agree to all tabulated digits.}
\label{tab: asymptotic D}
\end{table}


\section{\label{sec:Con}Conclusions}

We presented methods for constructing trapping boundaries in the phase spaces of dynamical systems. Although such boundaries have many uses, we have focused on boundaries that imply existence or nonexistence of heteroclinic connections, particularly in systems where those connections correspond to one-dimensional travelling waves in PDEs of reaction--diffusion type. For such RD equations our methods can be used to determine whether or not travelling waves exist at a specified speed. In many RD equations with monostable reaction terms, travelling waves exist if and only if their speed exceeds some minimum value $c_*$. In such cases, verifying existence of a travelling wave at some larger speed gives an upper bound on $c_*$, and verifying nonexistence at some smaller speed gives a lower bound on $c_*$. Crucially, our methods can be carried out not only analytically but also computationally using polynomial optimization, at least when the dynamical systems being studied have polynomial right-hand sides. 

We illustrated our approach by using polynomial optimization to compute bounds on $c_*$ for two examples of scalar RD equations and for a two-component RD system. Many of the upper and lower bounds reported here are sharp or very close to sharp, unlike most analytical estimates. Rather than replacing analysis, these computations complement it. Our computational results have guided us to new analytical upper bounds on $c_*$ for both examples of scalar RD equations, and to approximations for the asymptotic scalings of $c_*$ in the two-component RD system. Moreover, similar bounding computations for a two-component Keller--Segel model have led to an analytical proof for the exact value of $c_*$, as reported in a separate work~\cite{Bramburger}.

Our methods apply broadly to scalar RD equations, and their extension to multicomponent RD systems is straightforward when showing \emph{non}existence of travelling waves. To show existence in multicomponent systems, various complications arise that likely preclude a single unified approach. The main reason is that the relevant dynamical systems have a phase space dimension of at least three, so trapping boundaries alone may not imply the existence of a heteroclinic connection. In some multicomponent RD systems the remainder of the argument is provided by a monotone quantity, as arises naturally in the two-component example of~\cref{sec:Multi}. Lacking an obvious monotone quantity, more sophisticated techniques are needed to verify the existence of heteroclinic connections. Past authors have used trapping regions in conjunction with techniques such as Lyapunov functions, the stable manifold theorem, and generalizations of Wazewski's theorem \cite{Dunbar,Dunbar2,Hsu,Huang,Li,Li2}. All of these approaches may benefit from computational tools for finding trapping boundaries, but in this first work we have restricted ourselves to RD equations where finding the right trapping boundaries is the main challenge.

A natural continuation of the present work is to study one-dimensional travelling waves of RD equations whose reaction terms are not monostable. All examples in the present work have monostable reaction terms, resulting in travelling waves at all speeds above some minimum. In such examples it was not hard to verify existence or nonexistence at speeds sufficiently larger or smaller than this minimum, respectively. Bistable reaction terms, on the other hand, often lead to travelling waves that exist only at some unique speed. A prominent example is the Nagumo equation, $u_t = u_{xx} + u(u - \mu)(1 - u)$ with $\mu \in (0,1)$, where a travelling wave connecting $u=1$ to $u=0$ exists at the unique speed $c_0 = \frac{1}{\sqrt{2}}(1 - 2\mu)$ \cite{Fife}. In other words, in a suitably defined dynamical system whose heteroclinic connections correspond to travelling waves of the Nagumo equation, a connection exists only when $c=c_0$. At any $c$ value larger or smaller our methods might be able to verify nonexistence of a connection, but this alone would not imply upper or lower bounds on $c_0$. Producing such bounds requires distinguishing between the geometry of phase space in the $c>c_0$ and $c<c_0$ cases. This can be done analytically for the Nagumo equation but not for many other bistable RD equations. Novel methods based on computational polynomial optimization may lead to progress in studying such equations, as they have here for monostable RD equations.


\section*{Acknowledgements} 
We thank Alexander Chernyavsky for performing numerical integration of the generalized Fisher--KPP equation to obtain the wave speeds included in \cref{FisherTable}. We also thank Ryan Goh and Jane MacDonald for suggestions to improve the manuscript. Both authors were partially supported by the NSERC Discovery Grants Program via awards RGPIN-2018-04263, RGPAS-2018-522657, and DGECR-2018-00371, and one of us (JB) was partially supported by a PIMS Postdoctoral Fellowship.


\end{document}